\documentclass{amsart}
\usepackage{amsmath}
\usepackage{amssymb}
\usepackage{amsfonts}
\usepackage{amsfonts}
\usepackage[all]{xy}
\newtheorem{thm}{Theorem}[section]
\newtheorem{cor}[thm]{Corollary}

\newtheorem{defn}[thm]{Definition}
\newtheorem{rem}[thm]{Remark}

\newtheorem{exm}{Example}

\numberwithin{equation}{section}
\begin{document}

\title{On weakly $\delta$-semiprimary ideals of commutative rings}
\author{Ayman Badawi, DENIZ SONMEZ AND GURSEL YESILOT}
	\address{Department of Mathematics  $\&$ Statistics, The American University of
Sharjah, P.O.  Box 26666, Sharjah, United Arab Emirates}\email{abadawi@aus.edu}
\address{Yildiz Technical University, Department of Mathematics, Davutpa\c{s}a-Istanbul, Turkey, }
\email{dnzguel@hotmail.com, gyesilot@yildiz.edu.tr}

\maketitle

\begin{abstract}
	Let $R$ be a commutative ring with $ 1 \neq 0$. We recall that a proper ideal $I$ of $R$ is called a semiprimary  ideal of $R$ if whenever $a,b\in R$  and  $ab \in I$, then
$a\in \sqrt{I}$ or $b\in \sqrt{I}$. We  say $I$ is a {\it weakly semiprimary ideal} of $R$ if whenever $a,b\in R$ and $0 \not = ab \in I$, then
$a\in \sqrt{I}$ or $b\in \sqrt{I}$. In this paper, we introduce a new class of ideals that is closely related to the class of (weakly) semiprimary ideals. Let $I(R)$ be the set of all ideals of $R$ and let $\delta: I(R) \rightarrow I(R)$ be a function. Then $\delta$ is called an  expansion function of ideals of $R$ if  whenever $L, I, J$ are ideals of $R$ with $J \subseteq I$, then $L \subseteq \delta(L)$ and $\delta(J) \subseteq \delta(I)$. Let $\delta$ be an expansion function of ideals of $R$. Then a proper ideal $I$ of $R$ (i.e., $I \not = R$) is called a ({\it $\delta$-semiprimary}) {\it weakly $\delta$-semiprimary} ideal of $R$  if ($ab \in I$) $0 \not = ab \in I$ implies $a \in \delta(I)$ or $b \in \delta(I)$. For example, let $\delta: I(R) \rightarrow I(R)$ such that $\delta(I) = \sqrt{I}$. Then $\delta$ is an expansion function of ideals of $R$ and hence a proper ideal $I$ of $R$ is a ($\delta$-semiprimary) weakly $\delta$-semiprimary ideal of $R$  if and only if $I$ is a (semiprimary) weakly semiprimary ideal of $R$. A number of results concerning weakly $\delta$-semiprimary ideals and examples of
weakly $\delta$-semiprimary ideals are given.
\end{abstract}
\maketitle
\section{ Introduction}
 We assume throughout this paper that all rings are commutative with $ 1\neq 0 $. Let $R$ be a commutative ring. An ideal $I$ of $R$ is said to be proper if $ I \neq R $. Let $I$ be a proper ideal of $R$. Then $ \sqrt{I}$ denotes the radical ideal of  $I$ (i.e., $\sqrt{I} = \{x \in R \mid x^n \in I$ for some positive integer $n \geq 1\}$). Note that $\sqrt{\{0\}}$ is the set (ideal) of all nilpotent elements of $R$.

Let $I$ be a proper ideal of $R$. We recall from \cite{1} and \cite{6} that $ I $ is said to be \textit{weakly semiprime} if $ 0\neq x^{2}\in I $ implies $ x\in I $. We recall from \cite{1} (\cite{4}) that a proper ideal  $ I $  of $ R $ is said to be \textit{weakly prime} (\textit{weakly primary}) if $ 0\neq ab\in I $, then  $ a\in I $ or $ b\in I $ ($ a\in I $ or $ b\in \sqrt{I} $). Over the past several years, there has been considerable attention in the literature to prime ideals and their generalizations (for example: see \cite{1} -- \cite{10}, and \cite{12}).

Recall that a proper ideal $ I $ of a ring $R$ is called \textit{semiprimary} if whenever $ x,y\in R $ and $ xy\in I $, then $ x\in \sqrt{I} $ or $ y\in \sqrt{I} $. Gilmer in \cite{11} studied rings in which semiprimary ideals are primary.  In this paper, we define a proper ideal $ I $ of $ R $ to be \textit{weakly semiprimary} if whenever $ x,y\in R $ and $ 0\neq xy\in I $, then  $ x\in \sqrt{I} $ or $ y\in \sqrt{I} $. In fact, we will study a more general concept. Let $I(R)$ be the set of all ideals of $R$. Zhao \cite{12} introduced the concept of expansion of ideals of $R$. We recall from \cite{12} that a function $\delta : I(R) \rightarrow I(R)$ is called an expansion function of ideals of $R$ if whenever $L, I, J$ are ideals of $R$ with $J \subseteq I$, then $L \subseteq \delta(L)$ and $\delta(J) \subseteq \delta(I)$. Recall from \cite{12} that a proper ideal $I$ of $R$ is said to be a $\delta$-primary ideal of $R$ if whenever $ a, b \in  R $  with $ ab \in I$ implies $a \in I$ or $ b \in \delta(I)$, where $\delta$ is an expansion function of ideals of $R$. Let $\delta$ be an expansion function of ideals of $R$. In this paper, a proper ideal $I$ of $R$ (i.e., $I \not = R$) is called a ({\it $\delta$-semiprimary}) {\it weakly $\delta$-semiprimary} ideal of $R$  if ($ab \in I$) $0 \not = ab \in I$ implies $a \in \delta(I)$ or $b \in \delta(I)$. For example, let $\delta: I(R) \rightarrow I(R)$ such that $\delta(I) = \sqrt{I}$. Then $\delta$ is an expansion function of ideals of $R$ and hence a proper ideal $I$ of $R$ is a ($\delta$-semiprimary) weakly $\delta$-semiprimary ideal of $R$  if and only if $I$ is a (semiprimary) weakly semiprimary ideal of $R$. A number of results concerning weakly $\delta$-semiprimary ideals and examples of
weakly $\delta$-semiprimary ideals are given.

Let $\delta$ be an expansion function of ideals of a ring $R$. Among many results in this paper, it is shown (Theorem \ref{r3}) that if $I$ is a weakly $\delta$-semiprimary ideal of $R$ that is not $\delta$-semiprimary, then $I^2 = \{0\}$ and hence $I \subseteq \sqrt{\{0\}}$. If $I$ is a proper ideal of $R$ and $I^2 = \{0\}$, then $I$ need not be a weakly $\delta$-semiprimary ideal of $R$ (Example \ref{e1}). It is shown (Example \ref{e4}) that if $I, J$ are weakly $\delta$-semiprimary ideals of $R$ such that $\delta(I) = \delta(J)$ and $I + J \not = R$, then $I + J$ need not be a weakly $\delta$-semiprimary ideal of $R$. It is shown (Theorem \ref{r5.1}) that if $R$ is a Boolean ring, then every weakly semiprimary ideal of $R$ is weakly prime. It is shown (Theorem \ref{s3.1}) that if $S$ is a multiplicatively closed subset of $R$ such that $S \cap Z(R) = \emptyset$ (where $Z(R)$ is the set of all zerodivisor elements of $R$) and  $I$ is a weakly semiprimary ideal of $R$ such that $S \cap \sqrt{I} = \emptyset$, then $I_S$ is a weakly semiprimary ideal of $R_S$. It is shown (Corollary \ref{s3.3}) that if $I$ is a weakly semiprimary ideal of $R$ and $\frac{J}{I}$ is a weakly semiprimary ideal of $\frac{R}{I}$, then $J$ is a weakly semiprimary ideal of $R$. It is shown (Corollary \ref{s4.3}) that if $R = R_1\times R_2$, where $R_1, R_2$ are some rings with $1 \not = 0$, and $I$ is a proper ideal of $R$, then $I$ is a weakly semiprimary ideal of $R$ if and only if $I = \{(0, 0)\}$ or $I$  is a semiprimary ideal of $R$. It is shown (Theorem \ref{r9}) that if $I$ is a weakly $\delta$-semiprimary ideal of $R$ and $\{0\} \not = AB \subseteq I$ for some ideals $A, B$ of $R$, then $A\subseteq \delta(I)$ or $B \subseteq \delta(I)$.

\section{weakly $\delta$-semiprimary ideals}

\begin{defn}
  Let $I(R)$ be the set of all ideals of $R$.  We recall from \cite{12} that a function $\delta : I(R) \rightarrow I(R)$ is called an expansion function of ideals of $R$ if whenever $L, I, J$ are ideals of $R$ with $J \subseteq I$, then $L \subseteq \delta(L)$ and $\delta(J) \subseteq \delta(I)$.
 \end{defn}

 In the following example, we give some expansion functions of ideals of a ring $R$.

 \begin{exm} (\cite{7.1})
  Let $\delta: I(R)\rightarrow I(R)$ be a function. Then
 \begin{enumerate}
  \item If $\delta(I) = I$ for every ideal $I$ of $R$, then $\delta$ is an expansion function of ideals of $R$.
  \item If $\delta(I) = \sqrt{I}$ (note that $\sqrt{R} = R$) for every ideal $I$ of $R$, then $\delta$ is an expansion function of ideals of $R$.
  \item  Suppose that $R$ is a quasi-local ring (i.e., $R$ has exactly one maximal ideal) with maximal ideal $M$. If $\delta(I) = M$ for every proper ideal $I$ of $R$, then $\delta$ is an expansion function of ideals of $R$.
   \item Let $I$ be a proper ideal of $R$. Recall from \cite{11.1} that an element $r \in R$ is called integral over $I$ if there is an integer $n \geq 1$ and $a_i \in I^i$, $i = 1, ..., n$, $r^n + a_1r^{n-1} + a_2r^{n-2} + \cdots + a_{n-1}r + a_n = 0$. Let $\overline{I} = \{r \in R | r$ is integral over $I \}$.  Let $I \in I(R)$. It is known (see \cite{11.1}) that $\overline{I}$ is an ideal of $R$ and $I \subseteq \overline{I} \subseteq \sqrt{I}$ and if $J \subseteq I$, then $\overline{J} \subseteq \overline{I}$. If $\delta(I) = \overline{I}$ for every ideal $I$ of $R$, then $\delta$ is an expansion function of ideals of $R$.
    \item Let $J$ be a proper  ideal of $R$. If $\delta(I) = I + J$ for every ideal $I$ of $R$, then $\delta$ is an expansion function of ideals of $R$.
        \item Assume that $\delta_1, \delta_2$ are expansion functions of ideals of $R$. Let $\delta : I(R) \rightarrow I(R)$ such that $\delta(I) = \delta_1(I) + \delta_2(I)$. Then $\delta$ is an expansion function of ideals of $R$.
            \item Assume that $\delta_1, \delta_2$ are expansion functions of ideals of $R$. Let $\delta : I(R) \rightarrow I(R)$ such that $\delta(I) = \delta_1(I)\cap \delta_2(I)$. Then $\delta$ is an expansion function of ideals of $R$.
                \item Assume that $\delta_1, \delta_2$ are expansion functions of ideals of $R$. Let $\delta : I(R) \rightarrow I(R)$ such that $\delta(I) = (\delta_1 o \delta_2)(I) = \delta_1(\delta_2(I))$. Then $\delta$ is an expansion function of ideals of $R$.
\end{enumerate}
\end{exm}
We recall the following definitions.

\begin{defn}
Let $\delta$ be an expansion function of ideals of a ring $R$.
\begin{enumerate}
\item  A proper ideal $I$ of $R$ is called a ({\it $\delta$-semiprimary}) {\it weakly $\delta$-semiprimary} ideal of $R$ if whenever $ a, b \in R$ and ($ab \in I$) $0 \neq ab\in I$, then $ a\in \delta(I)$ or $b \in \delta(I)$.
    \item Recall that if  $\delta: I(R) \rightarrow I(R)$ such that $\delta(I) = \sqrt{I}$ for every proper ideal $I$ of $R$, then $\delta$ is an expansion function of ideals of $R$. In this case, a proper ideal $I$ of $R$ is called a ({\it semiprimary}) {\it weakly semiprimary} ideal of $R$ if whenever $ a, b \in R$ and ($ab \in I$) $0 \neq ab\in I$, then $ a\in \sqrt{I}$ or $b \in \sqrt{I}$.
    \item A proper ideal $I$ of $R$ is called a ({\it $\delta$-primary}) {\it weakly $\delta$-primary } ideal of $R$ if whenever $ a, b \in R$ and ($ab \in I$) $0 \neq ab \in I$, then $ a\in I$ or $b\in \delta(I)$.
\item A proper ideal $I$ of $R$ is called a {\it weakly prime} ideal of $R$ if whenever $a, b \in R$ and $0\neq ab\in I$, then $ a\in I$ or $b \in I$.
   \item  A proper ideal $I$ of $R$ is called a {\it weakly primary} ideal of $R$ if whenever $ a, b \in R$ and $0 \neq ab \in I$, then $a\in I$ or $b\in \sqrt{I}$.
   \end{enumerate}
      \end{defn}

We have the following trivial result, and hence we omit its proof.

\begin{thm}\label{r1} Let $ I $ be a proper ideal of $ R $ and let $\delta$ be an expansion function of ideals of $R$. Then
\begin{enumerate}
\item If $I$ is a $\delta$-primary ideal of $R$, then $I$ is a weakly $\delta$-semiprimary ideal of $R$. In particular, if  $ I $ is a primary ideal of $R$, then  $ I $ is a weakly semiprimary ideal of $R$.

\item If $I$ is a weakly $\delta$-primary ideal of $R$, then $I$ is a weakly $\delta$-semiprimary ideal of $R$. In particular, if  $ I $ is a weakly primary ideal of $R$, then  $ I $ is a weakly semiprimary ideal of $R$.

\item If  $I$ is a $\delta$-semiprimary ideal of $R$, then  $I$ is a weakly $\delta$-semiprimary ideal of $R$.
\item $\sqrt{\{0\}}$ is a weakly prime ideal of $R$ if and only if $\sqrt{\{0\}}$ is a weakly semiprimary ideal of $R$.
\item If $I$ is a weakly prime ideal of $R$, then $I$ is a weakly semiprimary ideal of $R$.

\end{enumerate}
\end{thm}
The following is an example of a proper ideal of a ring $R$ that is a weakly semiprimary ideal of $R$ but it is neither  weakly primary  nor weakly prime.

\begin{exm}
Let $ A=Z_2[X,Y] $ where $ X,Y $ are indeterminates. Then $I = (Y^2, XY)A$ and $J = (Y^2, X^2Y^2)A$ are ideals of $A$. Set $R = A/J$. Then $L =  I/J$ is an ideal of $R$  and $\sqrt{L} = (Y, XY)A/J$.  Since $ 0\neq XY + J \in L $ and neither $ X + J\in \sqrt{L} $ nor $ Y + J \in L $, we conclude that  $L$ is not a weakly primary ideal of $R$. Since $0 + J \not = XY+J \in L$ but neither $X + J \in L$ nor  $Y + J \in L$, $L$ is not a weakly prime ideal of $R$. It is  easy to check that $L$ is a weakly semiprimary ideal of $R$.
\end{exm}
The following is an example of an ideal that is weakly semiprimary but not semiprimary.
\begin{exm}
 Let $R=Z_{36} $. Then  $I =  \{0\} $ is a weakly semiprimary ideal of $R$ by definition. Note that $\sqrt{I} = 6R$. Since $0 = 4\cdot 9 \in I$ but neither $4 \in \sqrt{I}$ nor $9 \in \sqrt{I}$, we conclude that $I$ is not a semiprimary ideal of $R$.
 \end{exm}

\begin{defn} Let $\delta$ be an expansion function of ideals of a ring $R$. Suppose that $ I  $ is  a weakly $\delta$-semiprimary ideal of $ R $  and $x \in R$. Then $ x $ is called a dual-zero element of $ I $ if $ xy=0$ for some $y \in R$  and neither $x \in \delta(I)$ nor $y \in \delta(I)$ (note that $y$ is also a dual-zero element of $I$).
\end{defn}
\begin{rem}
Let $\delta$ be an expansion function of ideals of a ring $R$. Note that if  $ I  $ is  a weakly $\delta$-semiprimary ideal of $ R $ that is not $\delta$-semiprimary, then $I$ must have a dual-zero element of $R$.
\end{rem}

\begin{thm}\label{r2}
Let $\delta$ be an expansion function of ideals of a ring $R$ and $I$ be a weakly $\delta$-semiprimary ideal of $R$. If $x\in R$ is a dual-zero element of $I$, then $xI = \{0\}$.
\end{thm}
\begin{proof} Assume that $x \in R$ is a dual-zero element of $I$. Then $xy = 0$ for some $y \in R$ such that neither $x \in \delta(I)$ nor $y \in \delta(I)$. Let $i \in I$. Then $x(y + i) = 0 + xi = xi \in I$. Suppose that $xi \not = 0 $. Since $0 \not = x(y + i) = xi \in I$ and $I$ is a weakly $\delta$-semiprimary ideal of $R$, we conclude that $x \in \delta(I)$ or $(y + i) \in \delta(I)$ and hence $x \in \delta(I)$ or $y \in \delta(I)$, a contradiction. Thus $xi = 0$.
\end{proof}

\begin{thm}\label{r3}
Let $\delta$ be an expansion function of ideals of a ring $R$ and $I$ be a weakly $\delta$-semiprimary ideal of $R$ that is not $\delta$-semiprimary. Then $I^2 = \{0\}$ and hence $I \subseteq \sqrt{\{0\}}$.
\end{thm}
\begin{proof}
Since $I$ is a weakly $\delta$-semiprimary ideal of $R$ that is not $\delta$-semiprimary, we conclude that $I$ has a dual-zero element $x \in R$. Since $xy = 0$ and neither $x \in \delta(I)$ nor $y \in \delta(I)$, we conclude that $y$ is a dual-zero element of $I$. Let $i, j \in I$. Then by Theorem \ref{r2}, we have $(x + i)(y + j) = ij \in I$. Suppose that $ij \not = 0$. Since $0 \not = (x + i)(y + j) = ij \in I$ and $I$ is a weakly $\delta$-semiprimary ideal of $R$, we conclude that $x + i \in \delta(I)$ or $y + j \in \delta(I)$ and hence $x \in \delta(I)$ or $y \in \delta(I)$, a contradiction. Thus $ij = 0$ and hence $I^2 = \{0\}$.
\end{proof}
In view of Theorem \ref{r3}, we have the following result.
\begin{cor}\label{r4}
Let $I$ be a weakly semiprimary ideal of $R$ that is not semiprimary. Then $I^2 = \{0\}$ and hence $I \subseteq \sqrt{\{0\}}$.
\end{cor}

The following example shows that a proper ideal $I$ of $R$ with the property $ I^{2}= \{0\} $ need not be a weakly semiprimary ideal of $R$.

 \begin{exm}\label{e1} Let $ R=Z_{12} $. Then $ I=\{0, 6\} $ is an ideal of $ R $ and $ I^{2}=\{0\}$. Note that $\sqrt{I} = I$. Since $0\not =  2\cdot 3 \in I$ and neither $2 \in \sqrt{I}$ nor $3 \in \sqrt{I}$, we conclude that $I$ is not a weakly semiprimary ideal of $R$.
  \end{exm}

\begin{thm}\label{r5}
Let $\delta$ be an expansion function of ideals of a ring $R$ and $I$ be a proper ideal of $R$. If $\delta(I)$ is a weakly prime of $R$, then $I$ is a weakly $\delta$-semiprimary ideal of $R$. In particular, if $\sqrt{I}$ is a weakly prime of $R$, then $I$ is a weakly semiprimary ideal of $R$.
\end{thm}
\begin{proof}
Suppose that $0 \not = xy \in I$ for some $x, y \in R$. Hence $0 \not = xy \in \delta(I)$. Since $\delta(I)$ is weakly prime, we conclude that $x \in \delta(I)$ or $y \in \delta(I)$. Thus $I$ is a weakly $\delta$-semiprimary ideal of $R$.
\end{proof}

 Note that if $I$ is  a weakly semiprimary ideal of a ring $R$ that is not semiprimary, then $\sqrt{I}$ need not be a weakly prime ideal of $R$. We have the following example.

 \begin{exm}\label{e2}
 $I = \{0\}$ is a weakly semiprimary ideal of $Z_{12}$. However, $\sqrt{I} = \{0, 6\}$ is not a weakly prime ideal of  $Z_{12}$. For $0 \not = 2\cdot 3 \in \sqrt{I}$, but neither $2 \in \sqrt{I}$  nor  $3 \in \sqrt{I}$.
 \end{exm}
\begin{rem}
Note that a weakly prime ideal of a ring $R$ is weakly semiprimary but the converse is not true. Let $R = \frac{Z_4[X]}{(X^3)}$. Then $\frac{(X^2)}{(X^3)}$ is an ideal of $R$. Since $0 \not = (X + (X^3))\cdot (X + (X^3)) = X^2 + (X^3) \in I$ but $X + (X^3) \notin I$, we conclude that $I$ is not a weakly prime ideal of $R$. Since $\sqrt{I} = \frac{(2, X)}{(X^3)}$ is a prime ideal of $R$, $I$ is a (weakly) semiprimary ideal of $R$.
 \end{rem}

 Let $R$ be a Boolean ring (i.e., $x^2 = x$ for every $x\in R$). Since $\sqrt{I} = I$ for every proper ideal $I$ of $R$, we have the following result.

 \begin{thm}\label{r5.1}
 Let $R$ be a Boolean ring and $I$ be a proper ideal of $R$. The following statements are equivalent.
\begin{enumerate}
 \item $I$ is a weakly semiprimary ideal of $R$.
 \item $I$ is a weakly prime ideal of $R$.
 \end{enumerate}
 \end{thm}

\begin{thm}\label{r6}
Let $\delta$ be an expansion function of ideals of a ring $R$ and $I$ be a weakly $\delta$-semiprimary ideal of $R$. Suppose that $\delta(I)= \delta(\{0\})$. The following statements are equivalent.
\begin{enumerate}
\item $I$ is not $\delta$-semiprimary.
\item $\{0\}$ has a dual-zero element of $R$.
\end{enumerate}
\end{thm}
\begin{proof}
{\bf $(1) \Rightarrow (2)$}. Since $I$ is a weakly $\delta$-semiprimary ideal of $R$ that is not $\delta$-semiprimary, there are $x, y \in R$ such that $xy = 0$ and neither $x \in \delta(I)$ nor $y \in \delta(I)$. Since $\delta(I) = \delta(\{0\})$, we conclude that $x$ is a dual-zero element of $\{0\}$.

{\bf $(2)\Rightarrow (1)$}. Suppose that $x$ is a dual-zero element of $\{0\}$. Since $\delta(I) = \delta(\{0\})$, it is clear that $x$ is a dual-zero element of $I$.
\end{proof}
In view of Theorem \ref{r6}, we have the following result.
\begin{cor}\label{r7}
Let $I \subseteq \sqrt{\{0\}}$ be a proper ideal of $R$ such that $I$ is a weakly semiprimary ideal of $R$. The following statements are equivalent.
\begin{enumerate}
\item $I$ is not semiprimary.
\item $\{0\}$ has a dual-zero element of $R$.
\end{enumerate}
\end{cor}
\begin{proof} Since $\delta : I(R) \rightarrow I(R)$ such that $\delta(I) = \sqrt{I}$ for every  proper ideal $I$ of $R$ is an expansion function of ideals of $R$, we have $\delta(I) = \delta(\{0\})$. Thus the claim is clear by Theorem \ref{r6}.
\end{proof}

We show that the hypothesis "$\delta(I) = \delta(\{0\})$" in Theorem \ref{r6} is crucial, i.e. the following is an example of an ideal $I$ of a ring $R$ such that  $I \subseteq \sqrt{\{0\}}$ and $\{0\}$ has a dual-zero element of $R$ but $I$ is a $\delta$-semiprimary ideal of $R$ for some expansion function $\delta$ of ideals of $R$.
\begin{exm}\label{e3}
Let $R = Z_8$, $\delta : I(R) \rightarrow I(R)$ such that $\delta(I) = \sqrt{I}$ for every nonzero proper ideal $I$ of $R$, and $\delta(\{0\}) = \{0\}$. Let $I = 4R$. Then $\delta(I) = \sqrt{I} = 2R$. It is clear that $I$ is a $\delta$-semiprimary ideal of $R$ and $2$ is a dual-zero element of $\{0\}$.
\end{exm}

\begin{thm}\label{r11}
Let $\delta$ be an expansion function of ideals of a ring $R$ and $I$ be a weakly $\delta$-semiprimary ideal of $R$. If $J \subseteq I$ and $\delta(J) = \delta(I)$, then $J$ is a weakly $\delta$-semiprimary ideal of $R$.
\end{thm}
\begin{proof}
Suppose that $0\not = xy \in J$ for some $x, y \in R$. Since $J \subseteq I$, we have $0 \not = xy \in I$. Since $I$ is a weakly $\delta$-semiprimary ideal of $R$, we conclude that $x \in \delta(I)$ or $y \in \delta(I)$. Since $\delta(I) = \delta(J)$, we conclude that $x \in \delta(J)$ or $y \in \delta(J)$. Thus $J$ is a weakly $\delta$-semiprimary ideal of $R$.
\end{proof}
 In view of Theorem \ref{r11}, we have the following result.
 \begin{cor}\label{r12}
 Let $I$ be a weakly semiprimary ideal of $R$ such that $I \subseteq \sqrt{\{0\}}$. If $J \subseteq I$, then $J$ is a weakly semiprimary ideal of $R$. In particular, if $L$ is an ideal of $R$, then $LI$ and $L\cap I$ are weakly semiprimary ideals of $R$. Furthermore, if $n \geq 1$ is a positive integer, then $I^n$ is a weakly semiprimary ideal of $R$.
 \end{cor}

 \begin{thm}\label{r12.1}
 Let $ {I_{i}},_{i\in J} $ be a collection of weakly semiprimary ideals of a ring $R$ that are not semiprimary. Then $ I=\bigcap I_{i} $ is a weakly semiprimary ideal of $R$.
\end{thm}
\begin{proof} Note that $\sqrt{I=\bigcap I_{i}} = \sqrt{I_1} = \sqrt{\{0\}}$ by Theorem \ref{r3}. Hence the result follows.
\end{proof}

If $I, J$ are weakly semiprimary ideals of a ring $R$ such that $\sqrt{I} = \sqrt{J}$ and $I + J \not = R$, then $I + J$  need not be a weakly semiprimary ideal of $R$. We have the following example.

\begin{exm}\label{e4}
Let $A = \mathbb{Z}_2[T, U, X, Y]$, $H = (T^2, U^2, XY + T + U, TU, TX, TY, UX, UY)A$ be an ideal of $A$, and $R = A/H$. Then by construction of $R$, $I = (TA + H)/H = \{0, T + H\}$ and $J = (UA + H)/H = \{0, U + H\}$ are weakly semiprimary ideals of $R$ such that $|I| = |J| = 2$ and $\sqrt{I} = \sqrt{J} = \sqrt{\{0\}}$ (in R) = $(T, U, XY)A/ H$. Let  $L = I + J = (H + (T, U)A)/H$. Then  $\sqrt{L} = \sqrt{\{0\}}$ (in $R$) and $L$  is not a weakly  semiprimary ideal of $R$. For $0 \not = X + H \cdot Y + H = XY + H$, $X + H \not \in L$, and  $Y + H  \not \in \sqrt{L}$.
\end{exm}

 \begin{thm}\label{r13}
  Let $\delta$ be an expansion function of ideals of $R$ such that $\delta(\{0\})$ is a $\delta$-semiprimary ideal of $R$  and $\delta(\delta(\{0\}))=\delta(\{0\})$. Then
  \begin{enumerate}
  \item $\delta(\{0\})$ is a prime ideal of $R$.
    \item Suppose that  $I$ be a weakly $\delta$-semiprimary ideal of $R$. Then $I$ is a $\delta$-semiprimary ideal of $R$.
      \end{enumerate}
\end{thm}
\begin{proof}
\begin{enumerate}
\item Suppose that $ab \in \delta(\{0\})$ for some $a, b \in R$. Suppose that $a \notin \delta(\delta(\{0\})) = \delta(\{0\})$. Since $\delta(\{0\})$ is a $\delta$-semiprimary ideal of $R$ and $a\notin \delta(\delta(\{0\}))$, we have $b \in \delta(\delta(\{0\})) = \delta(\{0\})$. Thus $\delta(\{0\})$ is a prime ideal of $R$.
  \item Suppose that $I$ is not $\delta$-semiprimary. Clearly, $\delta(\{0\}) \subseteq \delta(I)$. Since  $I^{2} = \{0\}$ by Theorem \ref{r3} and $\delta(\{0\})$ is a prime ideal of $R$, we have $I \subseteq \delta(\{0\})$. Since $\delta(\delta(\{0\})) = \delta(\{0\})$, we have  $\delta(I) \subseteq \delta(\delta(\{0\})) = \delta(\{0\})$. Since $\delta(\{0\}) \subseteq \delta(I)$ and $\delta(I) \subseteq \delta(\{0\})$, we have $\delta(I) = \delta(\{0\})$  is a prime ideal of $R$. Since $\delta(I)$ is prime, $I$ is a $\delta$-semiprimary ideal of $R$, which is a contradiction.
  \end{enumerate}
\end{proof}
\begin{thm}\label{r14}
  Let $\delta$ be an expansion function of ideals of $R$ such that $\delta(\{0\})$ is a weakly $\delta$-semiprimary ideal of $R$, $\sqrt{\{0\}} \subseteq \delta(\{0\})$, and $\delta(\delta(\{0\}))=\delta(\{0\})$. Then
  \begin{enumerate}
  \item $\delta(\{0\})$ is a weakly prime ideal of $R$.
    \item Suppose that  $I$ is a weakly $\delta$-semiprimary ideal of $R$ that is not $\delta$-semiprimary. Then $\delta(I) = \delta(\{0\}) = \delta(\sqrt{\{0\}})$ is a weakly prime ideal of $R$ that is not prime. Furthermore, if $J \subseteq \sqrt{\{0\}}$, then $J$ is a weakly $\delta$-semiprimary ideal of $R$ that is not $\delta$-semiprimary and $\delta(J) = \delta(\{0\})$.
      \end{enumerate}
\end{thm}
\begin{proof}
\begin{enumerate}
\item Suppose that $0 \not =  ab \in \delta(\{0\})$ for some $a, b \in R$. Suppose that $a \notin \delta(\delta(\{0\})) = \delta(\{0\})$. Since $\delta(\{0\})$ is a weakly $\delta$-semiprimary ideal of $R$ and $a\notin \delta(\delta(\{0\}))$, we have $b \in \delta(\delta(\{0\})) = \delta(\{0\})$. Thus $\delta(\{0\})$ is a weakly prime ideal of $R$.
  \item Suppose that $I$ is not $\delta$-semiprimary. Then $I^2 = \{0\}$ by Theorem \ref{r3} and hence $I \subseteq \sqrt{\{0\}}$. Let $J$ be an ideal of $R$ such that $J \subseteq \sqrt{\{0\}}$. Since $\sqrt{\{0\}} \subseteq \delta(\{0\})$, we have $J \subseteq \delta(\{0\})$. Hence $\delta(J) \subseteq \delta(\delta(\{0\})) = \delta(\{0\})$. Since $\delta(\{0\}) \subseteq \delta(J)$ and $\delta(J) \subseteq \delta(\{0\})$, we conclude that $\delta(J) = \delta(\{0\})$. In particular, $\delta(I) = \delta(\{0\}) = \delta(\sqrt{\{0\}})$  is a weakly prime ideal of $R$. Since $\delta(\{0\})$ is a weakly $\delta$-semiprimary of $R$  and $\delta(J) = \delta(\{0\})$, we conclude that $J$ is a weakly $\delta$-semiprimary ideal of $R$. Since $I$ is not $\delta$-semiprimary, we conclude that  $\delta(I) = \delta(\{0\})$ is not a prime ideal of $R$. Since $\delta(J) = \delta(\{0\})$ is a weakly prime ideal of $R$ that is not prime, we conclude that $J$ is a weakly $\delta$-semiprimary ideal of $R$ that is not $\delta$-semiprimary.
  \end{enumerate}
\end{proof}
\section{Weakly $\delta$-semiprimary ideals under localization and ring-homomorphism}
For a ring $R$, let $Z(R)$ be the set of all zerodivisors of $R$.
\begin{thm}\label{s3.1} Let $S$ be a multiplicatively closed subset of $R$ such that $S \cap Z(R) = \emptyset$. If $I$ is a weakly semiprimary ideal of $R$ and $S \cap \sqrt{I} = \emptyset$, then $I_S$  is a weakly semiprimary ideal of $R_S$.
\end{thm}
\begin{proof}
Since $S \cap \sqrt{I} = \emptyset$, we conclude that $\sqrt{I_S} = (\sqrt{I})_S$.
Let $a,b\in R$, $s,t\in S$ such that $0\not=\frac{a}{s}\frac{b}{t} \in I_S$. Then there exists $u\in S$ such that $0\not=uab\in
I$. Since $u \in S$ and $S \cap \sqrt{I} = \emptyset$, we conclude that $0 \not = ab \in \sqrt{I}$. Since $I$ is a weakly semiprimary ideal of $R$, we conclude that $a \in \sqrt{I}$ or $b \in \sqrt{I}$. Thus $\frac{a}{s} \in \sqrt{I_S}$ or $\frac{b}{t} \in \sqrt{I_S}$.   Thus $I_S$ is a weakly semiprimary ideal of $R_S$.
\end{proof}

\begin{thm}\label{s3.2}
   Let $\gamma$ be an expansion function of ideals of $R$ and let $ I, J $ be proper ideals of $R$ with $ I \subseteq J $. Let $\delta: I(\frac{R}{I}) \rightarrow I(\frac{R}{I})$ be an expansion  function of ideals of $S = \frac{R}{I}$ such that $\delta(\frac{L + I}{I}) = \frac{\gamma(L + I)}{I}$ for every $L \in I(R)$. Then the followings statements hold.
  \begin{enumerate}
      \item If $J$ is a weakly $\gamma$-semiprimary ideal of $R$, then $\frac{J}{I}$ is a weakly $\delta$-semiprimary ideal of $S$.
       \item If $I$ is a weakly $\gamma$-semiprimary ideal of $R$ and $\frac{J}{I}$ is a weakly $\delta$-semiprimary ideal of $S$, then $J$ is a a weakly $\gamma$-semiprimary ideal of $R$.
  \end{enumerate}
\end{thm}

\begin{proof} First observe that since $I \subseteq J$, we have $I \subseteq J \subseteq \gamma(J)$ and $\delta(\frac{J}{I}) = \frac{\gamma(J)}{I}$.
     \begin{enumerate}
    \item Assume that $ab \in J\setminus I$ for some $a, b \in R$. Then $0 \not = ab \in J$. Hence $a \in \gamma(J)$ or $b \in \gamma(J)$. Thus $a + I \in \frac{\gamma(J)}{I}$ or $b + I \in \frac{\gamma(J)}{I}$. Thus $\frac{J}{I}$ is a weakly $\delta$-semiprimary ideal of $S = \frac{R}{I}$.
        \item Since $I \subseteq J$, we have $\gamma(I) \subseteq \gamma(J)$. Assume that $0 \not = ab \in J$ for some $a, b \in R$. Assume $ab \in I$. Since $I$ is a weakly $\gamma$-semiprimary ideal of $R$, we have $a \in \gamma(I) \subseteq \gamma(J)$ or $b \in \gamma(I) \subseteq \gamma(J)$. Assume that $ab \in J\setminus I$. Hence $I \not = ab + I \in \frac{J}{I}$. Since $\frac{J}{I}$ is a weakly $\delta$-semiprimary ideal of $S$, we have $a + I \in \frac{\gamma(J)}{I}$ or  $b + I \in \frac{\gamma(J)}{I}$. Thus $a\in \gamma(J)$ or $b \in \gamma(J)$. Thus $J$ is a a weakly $\gamma$-semiprimary ideal of $R$.
\end{enumerate}
\end{proof}

In view of Theorem \ref{s3.2}, we have the following result.

\begin{cor}\label{s3.3}
   Let $ I, J $ be proper ideals of $R$ with $ I \subseteq J $. Then the followings statements hold.
  \begin{enumerate}
      \item If $J$ is a weakly semiprimary ideal of $R$, then $\frac{J}{I}$ is a weakly semiprimary ideal of $\frac{R}{I}$.
       \item If $I$ is a weakly semiprimary ideal of $R$ and $\frac{J}{I}$ is a weakly semiprimary ideal of $\frac{R}{I}$, then $J$ is a a weakly semiprimary ideal of $R$.
  \end{enumerate}
\end{cor}

\begin{thm}\label{s3.4}
Let $R$, $S$ be rings and $f: R \rightarrow S$ be a surjective ring-homomorphism. Then
\begin{enumerate}
\item If $I$ is a weakly semiprimary ideal of $R$ and $kernel(f) \subseteq I$, then $f(I)$ is a weakly semiprimary ideal of $S$.
\item If $J$ is a weakly semiprimary ideal of $S$ and $kernel(f)$ is a weakly semiprimary ideal of $R$, then $f^{-1}(J)$ is a weakly semiprimary ideal of $R$.
\end{enumerate}
\end{thm}
\begin{proof}
\begin{enumerate}
\item Since $I$ is a weakly semiprimary ideal of $R$ and $kernel(f) \subseteq I$, we conclude that $\frac{I}{kernel(f)}$ is  a weakly semiprimary ideal of $\frac{R}{kernel(f)}$ by Corollary \ref{s3.3}(1). Since $\frac{R}{kernel(f)}$ is ring-isomorphic to $S$, the result follows.

    \item Let $L = f^{-1}(J)$. Then $kernel(f) \subseteq L$. Since $\frac{R}{kernel(f)}$ is ring-isomorphic to $S$, we conclude that $\frac{L}{kernel(f)}$ is a weakly semiprimary ideal of $\frac{R}{kernel(f)}$. Since $kernel(f)$ is a weakly semiprimary ideal of $R$ and $\frac{L}{kernel(f)}$ is a weakly semiprimary ideal of $\frac{R}{kernel(f)}$, we conclude that $L = f^{-1}(J)$ is a weakly semiprimary ideal of $R$ by  Corollary \ref{s3.3}(2).
        \end{enumerate}
        \end{proof}

  \section{Weakly $\delta$-semiprimary ideals in  product of rings}
Let  $R_{1}$, ... , $R_{n}$, where $n \geq 2$, be commutative rings with $1 \not = 0$. Assume that $\delta_1, ... , \delta_n$  are  expansion functions of ideals of $R_1, ... , R_n$, respectively. Let $R = R_1 \times \cdots \times R_n$. We define a function $\delta_{\times}: I(R)\rightarrow I(R)$ such that $\delta_{\times}(I_1\times \cdots \times I_n) = \delta_1(I_1)\times \cdots \times \delta_n(I_n)$ for every $I_i \in I(R_i)$, where $1 \leq i \leq n$.  Then it is clear that $\delta_{\times}$ is an expansion function of ideals of $R$. Note that every ideal of $R$ is of the form $I_1\times \cdots \times I_n$, where each $I_i$ is an ideal of $R_i$, $1 \leq i \leq n$.

\begin{thm}\label{s4.1}
\label{t15} Let$\ R_{1}$ and $R_{2}$ be commutative rings with $1 \not = 0$, $R = R_{1}\times R_{2}$, $\delta_1, \delta_2$ be expansion functions of ideals of $R_1, R_2$, respectively. Let  $I$ be a proper ideal of $R_1$. Then the following statements are equivalent.
\begin{enumerate}
\item $I\times R_{2}$ is a weakly $\delta_{\times}$-semiprimary ideal of $R.$

\item $I\times R_{2}$ is a $\delta_{\times}$-semiprimary ideal of $R.$

\item $I$ is a $\delta_{1}$-semiprimary ideal of $R_1$.
\end{enumerate}
\end{thm}
\begin{proof}
{(1)$\Rightarrow $(2)}.  Let $J = I\times R_{2}$. Then $J^2 \not = \{(0,0)\}$.  Hence $J$ is a $\delta_{\times}$-semiprimary ideal of $R$ by Theorem \ref{r3}.

{(2)$\Rightarrow $(3)}. Suppose that $I$ is not a $\delta_1$-semiprimary ideal of $%
R_{1}.$ Then there exist $a,b\in R_{1}$ such that $ab\in I$, but neither $%
a \in \delta_1(I)$ nor $b \in \delta_1(I)$. Since $%
(a,1)(b,1) = (ab, 1)\in I\times R_{2},$  we have  $%
(a,1)\in \delta_{\times}(I\times R_{2})$ or $(b,1)\in \delta_{\times}(I\times R_2)$. It follows that  $a\in \delta_1(I)$ or $b\in \delta_1(I)$, a contradiction. Thus $I$ is a $\delta_1$-semiprimary ideal of $R_{1}.$

{(3)$\Rightarrow $(1)}.  Let $I$ be a $\delta_1$-semiprimary ideal of $R_{1}$. Then it is clear that
$I\times R_{2}$ is a (weakly) $\delta_{\times}$-semiprimary ideal of $R$.
\end{proof}

\begin{thm}\label{s4.2} Let$\ R_{1}$ and $R_{2}$ be commutative rings with $1 \not = 0$, $R = R_1\times R_2$, and $\delta_1, \delta_2$ be expansion functions of ideals of $R_1, R_2$, respectively such that $\delta_2(K) = R_2$ for some ideal $K$ of $R_2$ if and only if $K = R_2$. Let $I = I_1 \times I_2$  be a proper  ideal of $R$, where $I_1, I_2$ are some ideals of $R_1$ and $R_2$, respectively. Suppose that $\delta_1(I_1) \neq R_1$. The following statements are equivalent.
\begin{enumerate}
\item $I$ is a weakly $\delta_{\times}$-semiprimary ideal of $R.$
\item $I = \{(0, 0)\}$ or $I = I_1\times R_2$ is a $\delta_{\times}$-semiprimary ideal of $R$ (and hence $I_1$ is a $\delta_1$-semiprimary ideal of $R_1$).
\end{enumerate}
\end{thm}
\begin{proof}
{$(1) \Rightarrow (2)$}. Assume that $ \{(0, 0)\} \neq I =I_{1}\times I_{2} $ is a weakly $\delta_\times$-semiprimary ideal of $R$.  Then there exists
 $ (0,0)\neq (x,y)\in I $ such that $ x\in I_{1} $ and $ y\in I_{2} $. Since $I$ is a weakly $\delta_\times$-semiprimary ideal of $R$ and $ (0,0)\neq (x, 1)(1, y) = (x, y)\in I $, we conclude $(x, 1)\in \delta_{\times}(I)$ or $(1, y)\in \delta_{\times}(I)$. Since $\delta_1(I_1) \neq R_1$, we conclude that $(1, y) \notin \delta_\times(I)$. Thus $(x, 1) \in \delta_\times(I)$ and hence $1\in \delta_2(I_2)$. Since $1 \in \delta_2(I_2)$, we conclude that $\delta_2(I_2) = R_2$ and hence $I_2 = R_2$ by hypothesis. Thus $I = I_1\times R_2$ is a $\delta_\times$-semiprimary ideal of $R$ by Theorem \ref{s4.1}.

 {$(2)\Rightarrow (1)$}. No comments.
\end{proof}

\begin{cor}\label{s4.3} Let$\ R_{1}$ and $R_{2}$ be commutative rings with $1 \not = 0$ and  $R = R_1\times R_2$. Let $I$ be a proper ideal of $R$. The following statements are equivalent.
\begin{enumerate}
\item $I$ is a weakly semiprimary ideal of $R.$
\item $I = \{(0, 0)\}$ or $I$ is a semiprimary ideal of $R$.
\item $I = \{(0, 0)\}$ or $I = I_1 \times R_2$ for some semiprimary ideal $I_1$ of $R_1$ or $I = R_1 \times I_2$ for some semiprimary ideal $I_2$ of $R_2$.
\end{enumerate}
\end{cor}
\section{Strongly weakly $\delta$-semiprimary ideals}
\begin{defn}
Let $\delta$ be an expansion function of ideals of a ring $R$. A proper ideal $I$ of $R$ is called a {\it strongly weakly $\delta$-semiprimary ideal} of $R$ if whenever $\{0\} \not = AB \subseteq I$ for some ideals $A, B$ of $R$, then $A\subseteq \delta(I)$ or $B \subseteq \delta(I)$. Hence, a proper ideal $I$ of $R$ is called a {\it strongly weakly semiprimary ideal} of $R$ if whenever $\{0\} \not = AB \subseteq I$ for some ideals $A, B$ of $R$, then $A \subseteq \sqrt{\{0\}}$ or $B \subseteq \sqrt{\{0\}}$.
\end{defn}

\begin{rem}
Let $\delta$ be an expansion function of ideals of a ring $R$. It is clear that a strongly weakly $\delta$-semiprimary ideal of $R$ is a weakly $\delta$-semiprimary ideal of $R$. In this section, we show that a proper ideal $I$ of $R$ is a strongly weakly $\delta$-semiprimary ideal of $R$ if and only if $I$ is a weakly $\delta$-semiprimary ideal of $R$.
\end{rem}

\begin{thm}\label{r8}
Let $\delta$ be an expansion function of ideals of a ring $R$ and $I$ be a weakly $\delta$-semiprimary ideal of $R$. Suppose that $AB \subseteq I$ for some ideals $A, B$ of $R$ and suppose that $ab = 0$ for some $a \in A$ and $b\in B$  such that  neither $a \in \delta(I)$ nor $b \in \delta(I)$. Then $AB = \{0\}$.
\end{thm}
\begin{proof}
First we show that $aB = bA = \{0\}$. Suppose that $aB \not = \{0\}$. Then $0 \not = ac \in I$ for some $c \in B$. Since $I$ is a weakly $\delta$-semiprimary ideal of $R$ and $a\not \in \delta(I)$, we conclude that $c \in \delta(I)$.  Hence $0\not = a(b + c) = ac \in I$. Thus $a \in \delta(I)$ or $(b + c) \in \delta(I)$. Since $c \in \delta(I)$, we conclude that $a \in \delta(I)$ or $b \in \delta(I)$, a contradiction. Thus $aB = \{0\}$. Similarly, $bA = \{0\}$. Now suppose that $AB \not = \{0\}$. Then there is an element $d \in A$ and there is an element $e \in B$ such that $0 \not = de \in I$. Since $I$ is a weakly $\delta$-semiprimary ideal of $R$, we conclude that $d \in \delta(I)$ or $e \in \delta(I)$. We consider three cases: {\bf Case I}. Suppose that $d \in \delta(I)$ and $e \not \in \delta(I)$. Since $aB = \{0\}$, we have $0 \not = e(d + a) = ed \in I$, we conclude that $e \in \delta(I)$ or $(d + a) \in \delta(I)$. Since $d \in \delta(I)$, we have $e \in \delta(I)$ or $a \in \delta(I)$, a contradiction. {\bf Case II}. Suppose that $d \not \in \delta(I)$ and $e \in \delta(I)$. Since $bA = \{0\}$, we have $0 \not = d(e + b) = de \in I$, we conclude that $d \in \delta(I)$ or $(e + b) \in \delta(I)$. Since $e \in \delta(I)$, we have $d \in \delta(I)$ or $b \in \delta(I)$, a contradiction. {\bf Case III}. Suppose that $d, e \in \delta(I)$. Since $aB = bA = \{0\}$, we have $0 \not =(b + e)(d + a) = ed \in I$, we conclude that $b + e \in \delta(I)$ or $d + a \in \delta(I)$. Since $d, e \in \delta(I)$, we have $b \in \delta(I)$ or $a \in \delta(I)$, a contradiction.  Thus $AB = \{0\}$.
\end{proof}

 \begin{thm}\label{r9}
 Let $\delta$ be an expansion function of ideals of a ring $R$ and $I$ be a weakly $\delta$-semiprimary ideal of $R$. Suppose that $\{0\} \not = AB \subseteq I$ for some ideals $A, B$ of $R$. Then $A\subseteq \delta(I)$ or $B \subseteq \delta(I)$ (i.e., $I$ is a strongly weakly $\delta$-semiprimary ideal of $R$).
 \end{thm}
 \begin{proof} Since $AB \not = \{0\}$, by Theorem \ref{r8} we conclude that whenever $ab \in I$ for some $a \in A$ and $b \in B$, then $a \in \delta(I)$ or $b \in \delta(I)$. Assume that $\{0\}\neq AB\subseteq I $ and $ A \not\subseteq \delta(I) $. Then there is an $x\in A\setminus\delta(I) $. Let $y\in B$. Since $xy\in AB\subseteq I$ and $\{0\} \not = AB$ and $x\not \in \delta(I)$, we conclude that $y \in \delta(I)$ by Theorem \ref{r8}.  Hence $ B\subseteq\delta(I)$.
 \end{proof}
 In view of Theorem \ref{r9}, we have the following result.
 \begin{cor}\label{r10}
 Let $I$ be a weakly semiprimary ideal of $R$. Suppose that $\{0\} \not = AB \subseteq I$ for some ideals $A, B$ of $R$. Then $A\subseteq \sqrt{I}$ or $B \subseteq \sqrt{I}$ (i.e., $I$ is a strongly weakly semiprimary ideal of $R$).
 \end{cor}

{\bf Acknowledgement}.

 The authors are grateful to the referee for the great effort in proofreading the manuscript.

\end{document}